\documentclass[reqno]{amsart}
 
\usepackage{graphicx}
\usepackage{amscd}
\usepackage{amsmath, amssymb}
\usepackage{graphics}
\usepackage{graphicx}
\usepackage{epsfig}
\usepackage{xcolor}
\usepackage{soul}
\usepackage{tikz}

\newtheorem{theorem}{Theorem}

\newtheorem{definition}[theorem]{Definition}

\newtheorem{question}[theorem]{Question}

\begin{document}

\title[Algebraic and topological aspects of $ST_n$ and its representations]{Algebraic and topological aspects of the singular twin group and its representations}

\author{Mohamad N. Nasser}
\address{Mohamad N. Nasser\\
         Department of Mathematics and Computer Science\\
         Beirut Arab University\\
         P.O. Box 11-5020, Beirut, Lebanon}         
\email{m.nasser@bau.edu.lb}

\author{Nafaa Chbili}
\address{Nafaa Chbili\\
Department of Mathematical Sciences\\
College of Science\\
United Arab Emirates University \\
15551 Al Ain, U.A.E.}
\email{nafaachbili@uaeu.ac.ae}

\begin{abstract}
In this article, we introduce the singular twin monoid and its corresponding group, constructed from both algebraic and topological perspectives. We then classify all complex homogeneous $2$-local representations of this constructed group. Moreover, we study the irreducibility of these representations and provide clear conditions under which irreducibility holds. Our results give a structured approach to understanding this new algebraic object and its representations.
\end{abstract}

\maketitle

\renewcommand{\thefootnote}{}



\vspace*{-0.4cm}

\section{Introduction} %%%%%%%%%% Section 1 %%%%%%%%%%%%%

The braid group on $n$ strings, denoted by $B_n$, was first presented by E. Artin in 1926 \cite{Artin1}. One way to picture $B_n$ is as $n$ parallel strings hanging in three-dimensional space, which may twist around one another as they descend but never cross. The group $B_n$ is generated by $n-1$ elements denoted by $\sigma_1, \sigma_2, \ldots, \sigma_{n-1}$, where each $\sigma_i$ represents the crossing of the $i$-th strand over the $(i+1)$-st strand, see Figure \ref{figure1}. These generators satisfy the classical braid relations, which encode the notion that different ways of performing local crossings can lead to equivalent overall braids. Beyond its algebraic structure, $B_n$ plays a central role in low-dimensional topology, since Alexander’s theorem shows that every knot can be expressed as the closure of a braid \cite{Alexander1923}.  

\vspace*{0.2cm}

A natural counterpart to the braid group, first presented by G. Shabat and V. Voevodsky, is the twin group on $n$ strings, denoted by $T_n$ \cite{Shabat1990}. It can be viewed as a flattened version of $B_n$, where the distinction between over and under crossings is ignored. The generators $s_1, s_2, \ldots, s_{n-1}$ of $T_n$ still describe swaps between neighboring strands, but without recording which strand passes on top, reflecting the two-dimensional nature of the group, see Figure \ref{figure2}. Thus, $T_n$ preserves the essence of braiding while fitting into a simpler algebraic framework closely related to Coxeter groups. Furthermore, joining the ends of a twin produces a doodle, connecting twin groups to the study of planar curves without self-intersections \cite{FT}.

\vspace*{0.2cm}

Both groups, $B_n$ and $T_n$, are deeply linked to topology in different ways. The braid group $B_n$ is well known to be isomorphic to the fundamental group of the configuration space of $n$ distinct points in the plane, which explains its rich connections to both algebraic and geometric topology. This interpretation underlies its role in knot theory, where the passage from braids to knots via closures provides a powerful bridge between group theory and low-dimensional topology. On the other hand, the twin group $T_n$ corresponds to a more restricted configuration space in which over/under distinctions have been erased. Its uses are mainly combinatorial, since $T_n$ gives a simplified model of braiding. It captures some of the planar features but loses part of the topological information that $B_n$ still preserves.   
 
\vspace*{0.2cm}

One of the important algebraic structures that extends $B_n$ is the singular braid monoid on $n$ strings, denoted by $SM_n$, which was presented first by J. Birman in \cite{J.Bir}. The monoid $SM_n$ is generated by the Artin generators $\sigma_1, \sigma_2, \ldots, \sigma_{n-1}$ of $B_n$ together with an additional family of singular generators denoted by $\tau_1, \tau_2, \ldots, \tau_{n-1}$. In \cite{R.Fe}, R. Fenn, E. Keyman, and C. Rourke showed that $SM_n$ embeds into a group, namely $SB_n$, called the singular braid group, which has the same generators and relations as $SM_n$. Both the monoid $SM_n$ and the group $SB_n$ enrich the study of braids by taking singularities into consideration, making them essential tools in several areas of mathematics. For more information on $SM_n$ and $SB_n$, see \cite{Das,B.Gem}.

\vspace*{0.2cm}

Group representations and their properties let us study the structure of a group from both algebraic and geometric perspectives. In this setting, abstract group elements are realized as linear transformations, which makes their structure more concrete. One of the important properties to be studied for a representation is its irreducibility. A representation is said to be irreducible if it has no nontrivial subrepresentations, and otherwise it is called reducible. Irreducible representations are particularly significant since they serve as the basic building blocks of representation theory. More deeply, constructing such representations is fundamental in various fields, including quantum mechanics and particle physics \cite{J.Yang}.

\vspace*{0.2cm}

One famous type of representation is the $k$-local representation which was presented in \cite{Nas20241}. For a group $G$ with generators $g_1, g_2, \dots, g_{n-1}$, a representation of $G$ into $\mathrm{GL}_n(\mathbb{Z}[t^{\pm 1}])$ is said to be $k$-local if each generator $g_i$ is mapped to a block matrix of the form
\[
\begin{pmatrix}
I_{i-1} & 0 & 0 \\
0 & M_i & 0 \\
0 & 0 & I_{n-i-1}
\end{pmatrix},
\] 
where $M_i \in \mathrm{GL}_k(\mathbb{Z}[t^{\pm 1}])$ and $I_r$ is the $r \times r$ identity matrix. The $k$-local representations of the braid group, the singular braid monoid, and the twin group have been classified and studied for $k=2$ and $k=3$ \cite{Mik2013,Mayassi2025,Mayasi20251,M.Nass.3twin}.  

\vspace*{0.2cm}

The main objective of this paper is to construct and study a new group that extends $T_n$, in complete analogy with the relationship between the singular braid group $SB_n$ and the braid group $B_n$. This group, which is called the singular twin group and denoted by $ST_n$, is introduced and developed from both algebraic and topological perspectives (Section~3). The construction provides a natural framework that captures additional singular structures while preserving the fundamental properties of $T_n$. Once the group is established, we proceed to classify all complex homogeneous $2$-local representations of $ST_n$ (Section 4). In order to gain a deeper understanding of the representation of this group, we further examine the irreducibility of these representations and determine precise conditions under which they become irreducible (Section 5). Lastly, we give some open topics to be studied as future work (Section 6).

\vspace*{0.1cm}

\section{Main definitions and previous results} %%%%%%%%%% Section 2 %%%%%%%%%%%%%

In this section, we present the main group and monoid structures and presentations relevant to our study. We begin with the braid group $B_n$ and its normal subgroup, the pure braid group $P_n$.

\begin{definition} \cite{Artin1, Artin2}
The braid group on $n$ strands, denoted by $B_n$, is a discrete group generated by $\sigma_1,\sigma_2,\ldots,\sigma_{n-1}$ that satisfy the following relations.
\begin{equation} \label{eqs1}
\ \ \ \ \sigma_i\sigma_{i+1}\sigma_i = \sigma_{i+1}\sigma_i\sigma_{i+1} ,\hspace{0.45cm} i=1,2,\ldots,n-2,
\end{equation}
\begin{equation} \label{eqs2}
\sigma_i\sigma_j = \sigma_j\sigma_i , \hspace{1.45cm} |i-j|\geq 2.
\end{equation} 
\end{definition}

\vspace*{0.1cm}
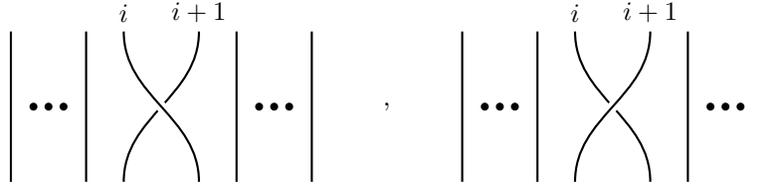
\begin{figure}[h!]
\begin{center}
\begin{tikzpicture}
 	\draw[thick] (-1.5,0)--(-1.5,2); 
    \fill (-1,1) circle(1.5pt) (-1.2,1) circle(1.5pt)(-0.8,1)circle(1.5pt);       
    \draw[thick] (-.5,0)--(-.5,2);
    \draw[thick] (2.5,0)--(2.5,2); 
    \fill (2,1) circle(1.5pt) (2.2,1) circle(1.5pt)
    (1.8,1)circle(1.5pt);       
    \draw[thick] (1.5,0)--(1.5,2);
    \draw[thick] (0,0) to[out=90, in=-90] (1,2);
    \draw[thick, white, line width=4pt] (1,0) to[out=90, in=-90] (0,2); % overlap to create crossing
    \draw[thick] (1,0) to[out=90, in=-90] (0,2);
    \node[above] at(0,2){$i$};
    \node[above] at(1,2){$i+1$};
	\node at (3.5,1){,};  
 	\draw[thick] (5.5,0)--(5.5,2); 
    \fill (5,1) circle(1.5pt) (4.8,1) circle(1.5pt)(5.2,1)circle(1.5pt);       
    \draw[thick] (4.5,0)--(4.5,2);
    \draw[thick] (8.5,0)--(8.5,2); 
    \fill (8,1) circle(1.5pt) (8.2,1) circle(1.5pt)(7.8,1)circle(1.5pt);       
    \draw[thick] (7.5,0)--(7.5,2);    
    \draw[thick] (7,0) to[out=90, in=-90] (6,2);
    \draw[thick, white, line width=4pt] (6,0) to[out=90, in=-90] (7,2); % overlap to create crossing
    \draw[thick] (6,0) to[out=90, in=-90] (7,2);
    \node[above] at(6,2){$i$};
    \node[above] at (7,2) {$i+1$};
   % \node at (3.5, -0.5) {The generators \(\sigma_i\) and \(\sigma_i^{-1}\).};
\end{tikzpicture}
\caption{The generators \(\sigma_i\) and \(\sigma_i^{-1}\).}
\label{figure1}
\end{center}
\end{figure}
\vspace*{0.1cm}

\begin{definition} \cite{Artin1, Artin2}
The pure braid group on $n$ strands, denoted by $P_n$, is defined as the kernel of the homomorphism $B_n \to S_n$ defined by $\sigma_i \mapsto (i \hspace{0.2cm} i+1)$, $1\leq i \leq n-1$, where $S_n$ is the symmetric group of $n$ elements. It admits a presentation with the following generators.
$$A_{ij}=\sigma_{j-1}\sigma_{j-2} \ldots \sigma_{i+1}\sigma^2_{i}\sigma^{-1}_{i+1} \ldots \sigma^{-1}_{j-2}\sigma^{-1}_{j-1}, \hspace{1cm} 1\leq i<j\leq n.$$
\end{definition}

\vspace*{0.1cm}

Next, we present the twin group $T_n$ along with its normal subgroup, the pure twin group $PT_n$.

\begin{definition}\cite{Shabat1990}
The twin group on $n$ strands, denoted by $T_n$, is a discrete group generated by $s_1,s_2,\ldots,s_{n-1}$ that satisfy the following relations.
\begin{equation} \label{eqs3}
\hspace{1.46cm} s_i^2=1,\hspace{0.95cm} i=1,2,\ldots,n-2,
\end{equation}
\begin{equation} \label{eqs4}
s_is_j=s_js_i, \hspace{0.55cm} |i-j|\geq 2.
\end{equation}
\end{definition}

\vspace*{0.1cm}
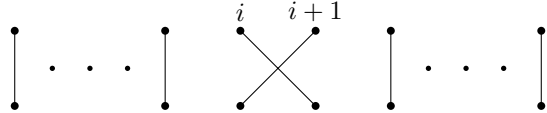
\begin{figure}[h!]
\begin{center}
\begin{tikzpicture}
\draw (0,0)--(0,1)(2,0)--(2,1)(5,0)--(5,1)(3,0)--(4,1)(4,0)--(3,1)(7,1)--(7,0);
\foreach \n in {0,2,3,4,5,7} \fill (\n,1)circle(1.5pt);
\foreach \n in {0,2,3,4,5,7} \fill (\n,0)circle(1.5pt);
\foreach \n in {.5,1,1.5,5.5,6,6.5} \fill (\n,0.5)circle(1pt);
\node[above] at (3,1) {$i$};
\node[above] at (4,1) {$i+1$};
%\node at (3.5, -0.5) {The generator $s_i$.};
\end{tikzpicture}
\end{center}
\caption{The generator $s_i$.}
\label{figure2}
\end{figure}
\vspace*{0.1cm}

\begin{definition} \cite{Khovanov1996}
The pure twin group on $n$ strands, denoted by $PT_n$, is defined as the kernel of the homomorphism $T_n \to S_n$ defined by $s_i \mapsto (i \hspace{0.2cm} i+1)$, $1\leq i \leq n-1$, where $S_n$ is the symmetric group of $n$ elements.
\end{definition}

\noindent In \cite{Bardakov2019}, Bardakov et. al found a generating set of $PT_n$ for $n>2$ using the Reidemeister–Schreier method explained in \cite{Mag}.

\vspace*{0.15cm}

We now move through two significant extensions of the braid group $B_n$: the singular braid monoid $SM_n$ and the singular braid group $SB_n$. Also, we introduce the singular pure braid group $SP_n$, which is a normal subgroup of $SB_n$.

\begin{definition}
\cite{J.Bir} The singular braid monoid, denoted by $SM_n$, is the monoid generated by the generators $\sigma_1^{\pm 1},\sigma_2^{\pm 1}, \ldots,\sigma_{n-1}^{\pm 1}$ of $B_n$ and the singular generators $\tau_1,\tau_2, \ldots, \tau_{n-1}$. The generators of \(SM_n\) satisfy the relations (\ref{eqs1}) and (\ref{eqs2}) of $B_n$ in addition to the following relations.
\begin{equation} \label{eqs8}
\tau_i\tau_j=\tau_j\tau_i ,\hspace{1.55cm} |i-j|\geq 2,
\end{equation}
\begin{equation} \label{eqs9}
\tau_i\sigma_j=\sigma_j\tau_i ,\hspace{1.5cm} |i-j|\geq 2,
\end{equation}
\begin{equation} \label{eqs10}
\hspace{1.15cm} \tau_i\sigma_i=\sigma_i\tau_i ,\hspace{1.55cm} i=1,2,\ldots ,n-1,
\end{equation}
\begin{equation} \label{eqs11}
\ \ \ \ \sigma_i\sigma_{i+1}\tau_i=\tau_{i+1}\sigma_i\sigma_{i+1}, \hspace{0.5cm} i=1,2,\ldots,n-2,
\end{equation}
\begin{equation} \label{eqs12}
\ \ \ \ \tau_{i}\sigma_{i+1}\sigma_{i}=\sigma_{i+1}\sigma_{i}\tau_{i+1}, \hspace{0.5cm} i=1,2,\ldots,n-2.
\end{equation}    
\end{definition}

\noindent By adjoining the inverses of the generators $\tau_i,\ 1\leq i \leq n-1$, we obtain an extension of $B_n$, generated by $\sigma_1,\sigma_2,\ldots,\sigma_{n-1}$ and $\tau_1,\tau_2,\ldots,\tau_{n-1}$, called the singular braid group and denoted by $SB_n$.

\vspace*{0.1cm}
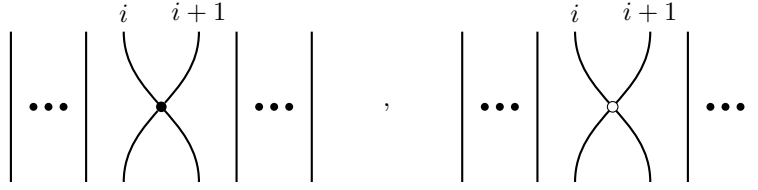
\begin{figure}[h!]
\begin{center}
\begin{tikzpicture}
	\draw[thick] (5.5,0)--(5.5,2); 
    \fill (5,1) circle(1.5pt) (4.8,1) circle(1.5pt)(5.2,1)circle(1.5pt);       
    \draw[thick] (4.5,0)--(4.5,2);
    \draw[thick] (8.5,0)--(8.5,2); 
    \fill (8,1) circle(1.5pt) (8.2,1) circle(1.5pt)(7.8,1)circle(1.5pt);       
    \draw[thick] (7.5,0)--(7.5,2);    
	\draw[thick] (7,0) to[out=90, in=-90] (6,2);
    \draw[thick, white, line width=4pt] (6,0) to[out=90, in=-90] (7,2); % overlap to create crossing
    \draw[thick] (6,0) to[out=90, in=-90] (7,2);
    \draw[draw=black,fill=white] (6.5,1) circle (2pt); % Bold intersection point    
    \node[above] at(6,2){$i$};
    \node[above] at (7,2) {$i+1$};
	\node at (3.5,1){,};    
 	\draw[thick] (-1.5,0)--(-1.5,2); 
    \fill (-1,1) circle(1.5pt) (-1.2,1) circle(1.5pt)(-0.8,1)circle(1.5pt);       
    \draw[thick] (-.5,0)--(-.5,2);
    \draw[thick] (2.5,0)--(2.5,2); 
    \fill (2,1) circle(1.5pt) (2.2,1) circle(1.5pt)(1.8,1)circle(1.5pt);       
    \draw[thick] (1.5,0)--(1.5,2);
    \draw[thick] (0,0) to[out=90, in=-90] (1,2);
    \draw[thick, white, line width=4pt] (1,0) to[out=90, in=-90] (0,2); % overlap to create crossing
    \draw[thick] (1,0) to[out=90, in=-90] (0,2);
    \fill[black] (0.5,1) circle (2pt); % Bold intersection point
    \node[above] at(0,2){$i$};
    \node[above] at(1,2){$i+1$};
   % \node at (3.5, -0.5) {The generators \(\tau_i\) and \(\tau_i^{-1}\).};
\end{tikzpicture}
\caption{The generators \(\tau_i\) and \(\tau_i^{-1}\).}
\label{figure3}
\end{center}
\end{figure}
\vspace*{0.1cm}

\begin{definition} \cite{Das}
The singular pure braid group on $n$ strands, denoted by $SP_n$, is defined as the kernel of the homomorphism $SB_n \to S_n$ defined by $\sigma_i \mapsto (i \hspace{0.2cm} i+1)$ and $\tau_i \mapsto (i \hspace{0.2cm} i+1)$, $1\leq i \leq n-1$, where $S_n$ is the symmetric group of $n$ elements.
\end{definition}

\noindent Similarly to the case of the pure twin group, Bardakov et al. \cite{Bar2025} obtained a generating set of $SP_n$ for $n>2$ by applying the Reidemeister–Schreier method described in \cite{Mag}.

\vspace*{0.15cm}

In the following, we introduce the idea of $k$-local representations for a group $G$ with finitely many generators.

\vspace*{0.1cm}

\begin{definition}\cite{Nas20241}
Let $G$ be a group generated by $g_1,g_2,\ldots,g_{n-1}$. A representation $\theta: G \to \mathrm{GL}_{m}(\mathbb{C})$ is called $k$-local if it takes the form
$$\theta(g_i) =
\begin{pmatrix}
I_{i-1} & 0 & 0 \\
0 & M_i & 0 \\
0 & 0 & I_{n-i-1}
\end{pmatrix}
\hspace*{0.2cm} \text{for} \hspace*{0.2cm} 1\leq i\leq n-1,$$ 
where $M_i \in \mathrm{GL}_k(\mathbb{C})$ with $k = m-n+2$, and $I_r$ denotes the $r \times r$ identity matrix. The representation is called homogeneous if all the matrices $M_i$ coincide.
\end{definition}

\vspace*{0.1cm}
 
In recent years, research on $k$-local representations has made steady progress. Mikhalchishina first classified the $2$-local representations of $B_3$ and all complex homogeneous $2$-local representations of $B_n$ for $n \geq 3$ \cite{Mik2013}. Later, Mayassi and Nasser studied the complex homogeneous $3$-local representations of $B_n$ for $n \geq 4$ \cite{Mayassi2025}. The following are two famous examples of $k$-local representations of the braid group $B_n$ with different degrees $k$.
 
\begin{definition} \cite{1} \label{defBurau}
Let $t$ be indeterminate. The Burau representation $\rho_B(t): B_n\to \mathrm{GL}_n(\mathbb{Z}[t^{\pm 1}])$ is the representation given by
$$\sigma_i\mapsto \left( \begin{array}{c|@{}c|c@{}}
   \begin{matrix}
     I_{i-1} 
   \end{matrix} 
      & 0 & 0 \\
      \hline
    0 &\hspace{0.2cm} \begin{matrix}
   	1-t & t\\
   	1 & 0\\
\end{matrix}  & 0  \\
\hline
0 & 0 & I_{n-i-1}
\end{array} \right) \hspace*{0.2cm} \text{for} \hspace*{0.2cm} 1\leq i\leq n-1.$$ 
\end{definition}

\vspace*{0.1cm}

\begin{definition} \cite{19} \label{Fdef}
Let $t$ be indeterminate. The $F$-representation $\rho_F(t): B_n \to \mathrm{GL}_{n+1}(\mathbb{Z}[t^{\pm 1}])$ is the representation given by
$$\sigma_i\mapsto \left( \begin{array}{c|@{}c|c@{}}
   \begin{matrix}
     I_{i-1} 
   \end{matrix} 
      & 0 & 0 \\
      \hline
    0 &\hspace{0.2cm} \begin{matrix}
   		1 & 1 & 0 \\
   		0 &  -t & 0 \\   		
   		0 &  t & 1 \\
   		\end{matrix}  & 0  \\
\hline
0 & 0 & I_{n-i-1}
\end{array} \right) \hspace*{0.2cm} \text{for} \hspace*{0.2cm} 1\leq i\leq n-1.$$ 
\end{definition}

\vspace*{0.1cm}

Regarding $k$-local representations of the twin groups, T. Mayassi and M. Nasser classified all $2$-local representations of $T_n$ for all $n\geq 2$ \cite{Mayasi20251}. Moreover, M. Nasser determined all $3$-local representations of the twin group $T_n$, the virtual twin group $VT_n$, and the welded twin group $WT_n$, for all $n\geq 4$ \cite{M.Nass.3twin}. On the other hand, in \cite{M.N.twin}, M. Nasser presented two particular representations of the twin group $T_n$ for $n\geq 2$. These representations are referred to as $N_1$ and $N_2$, respectively, and are described explicitly with their main results in the following.

\vspace*{0.1cm}
 
\begin{definition} \cite{M.N.twin}
The $N_1$-representation $\eta_1: T_n \rightarrow \mathrm{GL}_n(\mathbb{Z}[t^{\pm 1}])$, where $t$ is indeterminate, is the representation defined by
$$s_i\mapsto \left( \begin{array}{c|@{}c|c@{}}
   \begin{matrix}
     I_{i-1} 
   \end{matrix} 
      & 0 & 0 \\
      \hline
    0 &\hspace{0.2cm} \begin{matrix}
   	1-t & t\\
   	2-t & t-1\\
\end{matrix}  & 0  \\
\hline
0 & 0 & I_{n-i-1}
\end{array} \right) \text{ for } 1\leq i \leq n-1.$$ 
\end{definition}

\vspace*{0.1cm}

\begin{theorem} \cite{M.N.twin}
The representation $\eta_1: T_n \rightarrow \text{GL}_n(\mathbb{Z}[t^{\pm 1}])$ is reducible to the degree $n-1$ for all $n\geq 3$. Moreover, the complex specialization of its $(n-1)$-composition factor, namely $\eta_1': T_n \rightarrow \text{GL}_{n-1}(\mathbb{C})$, is irreducible if and only if $t\neq \frac{2n-2}{n-2}$ and $t\neq 2$.
\end{theorem}

\vspace*{0.1cm}
 
\begin{definition}\cite{M.N.twin}
The $N_2$-representation $\eta_2: T_n \rightarrow \mathrm{GL}_n(\mathbb{Z}[t^{\pm 1}])$, where $t$ is indeterminate, is the representation defined by
$$s_i \mapsto \left( \begin{array}{c|@{}c|c@{}}
   \begin{matrix}
     I_{i-1} 
   \end{matrix} 
      & 0 & 0 \\
      \hline
    0 &\hspace{0.2cm} \begin{matrix}
   	0 & f(t)\\
   	f^{-1}(t) & 0\\
\end{matrix}  & 0  \\
\hline
0 & 0 & I_{n-i-1}
\end{array} \right) \text{ for } 1\leq i \leq n-1,$$
where $f(t)\in \mathbb{Z}[t^{\pm 1}]$ with $f(t)$ is invertible in $\mathbb{Z}[t^{\pm 1}]$ in and $f^{-1}(t)=\frac{1}{f(t)}$.
\end{definition}

\vspace*{0.1cm}

\section{Algebraic and topological interpretation of the singular twin} %%%%%%%%%% Section 3 %%%%%%%%%%%%

In parallel with the singular braid monoid and the singular braid group, we introduce in this section the singular twin monoid and the singular twin group. These constructions provide the twin group counterparts of their braid analogues, serving as a foundation for studying their algebraic and topological structure and representations.

\subsection{Algebraic interpretation} 

In this subsection, we introduce the presentation of the \emph{singular twin monoid} and its corresponding group from algebraic perspective, followed by the definition of the \emph{singular pure twin group}.

\begin{definition}
The singular twin monoid, denoted by $STM_n$, is the monoid generated by the generators $s_1, s_2, \ldots,s_{n-1}$ of $T_n$ and the singular generators $\tau_1,\tau_2, \ldots, \tau_{n-1}$. The generators of \(STM_n\) satisfy the relations (\ref{eqs3}) and (\ref{eqs4}) of $T_n$ in addition to the following relations.
\begin{equation} \label{eqs13}
\tau_i\tau_j=\tau_j\tau_i ,\hspace{1.5cm} |i-j|\geq 2,
\end{equation}
\begin{equation} \label{eqs14}
\tau_is_j=s_j\tau_i ,\hspace{1.5cm} |i-j|\geq 2,
\end{equation}
\begin{equation} \label{eqs15}
\hspace{1.15cm} \tau_is_i=s_i\tau_i ,\hspace{1.55cm} i=1,2,\ldots ,n-1,
\end{equation}
\begin{equation} \label{eqs16}
\ \ \ \ s_is_{i+1}\tau_i=\tau_{i+1}s_is_{i+1}, \hspace{0.5cm} i=1,2,\ldots,n-2,
\end{equation}
\begin{equation} \label{eqs17}
\ \ \ \ \tau_{i}s_{i+1}s_{i}=s_{i+1}s_{i}\tau_{i+1}, \hspace{0.5cm} i=1,2,\ldots,n-2.
\end{equation}    
\end{definition}

\noindent Remark that relations \eqref{eqs16} and \eqref{eqs17} are equivalent as $s_i^2=1$. By adjoining the inverses of the generators $\tau_i,\ 1\leq i \leq n-1$, we obtain an extension of $T_n$, generated by $s_1,s_2,\ldots,s_{n-1}$ and $\tau_1,\tau_2,\ldots,\tau_{n-1}$, which we call the \emph{singular twin group} and we denote it by $ST_n$.

\vspace*{0.1cm}

\begin{definition} \cite{Das}
The singular pure  twin group on $n$ strands, denoted by $SPT_n$, is defined as the kernel of the homomorphism $ST_n \to S_n$ defined by $s_i \mapsto (i \hspace{0.2cm} i+1)$ and $\tau_i \mapsto (i \hspace{0.2cm} i+1)$, $1\leq i \leq n-1$, where $S_n$ is the symmetric group of $n$ elements.
\end{definition}

\vspace*{0.1cm}

\begin{question}
Give a presentation of $SPT_n$ for $n>2$ by generators and relations.
\end{question}

\vspace*{0.1cm}

The Reidemeister--Schreier method provides a way to obtain a presentation of $SPT_n$ by generators and relations. However, the computation in our case  becomes quite involved because of the large number of defining relations in the group. 
Therefore, in what follows we restrict our attention to the case $n=3$ and show that $SPT_3$ is generated by the following four elements:
\[
a:=s_1\tau_1,\qquad b:=s_2\tau_2,\qquad c:=s_2 a s_2,\qquad d:=s_1 b s_1.
\]
Note that the quotient $ST_3/SPT_3\cong S_3$.  Choose the standard Schreier transversal $\Lambda$ consisting of reduced words representing each permutation:
\[
\Lambda=\{\,\lambda_0=e,\;\lambda_1=s_1,\;\lambda_2=s_2,\;\lambda_3=s_1s_2,\;
\lambda_4=s_2s_1,\;\lambda_5=s_1s_2s_1\,\}.
\]

\noindent For each coset representative $\lambda_i$ and each parent generator $x\in\{s_1,s_2,\tau_1,\tau_2\}$ the Schreier generator is
\[
a_{\lambda_i,x}:=\lambda_i x\bigl(\overline{\lambda_i x}\bigr)^{-1},
\]
where $\overline{\lambda_i x}\in\Lambda$ is the chosen representative with the same image in $S_3$ as $\lambda_i x$.  Because $\pi(\tau_j)=\pi(s_j)$ we always have $\overline{\lambda\,\tau_j}=\overline{\lambda\,s_j}$, which simplifies many computations.

The nontrivial Schreier generators (after cancelling obvious trivial ones) are computed as follows.
\[
\begin{aligned}
& s_2 s_1 s_2\bigl(\overline{s_2 s_1 s_2}\bigr)^{-1}
   = s_2 s_1 s_2 s_1 s_2 s_1
   = (s_2 s_1)^3, \\[4pt]
& s_1 s_2 s_1 s_2\bigl(\overline{s_1 s_2 s_1 s_2}\bigr)^{-1}
   = s_1 s_2 s_1 s_2 s_2 s_1
   = (s_1 s_2)^3, \\[4pt]
& s_1 \tau_1\bigl(\overline{s_1 \tau_1}\bigr)^{-1}
   = s_1 \tau_1 s_1^2
   = s_1 \tau_1 = a, \\[4pt]
& s_2 \tau_1\bigl(\overline{s_2 \tau_1}\bigr)^{-1}
   = s_2 \tau_1 (s_2 s_1)^{-1}
   = s_2 \tau_1 s_1 s_2
   = s_2 s_1 \tau_1 s_2=c, \\[4pt]
& s_1 s_2 \tau_1\bigl(\overline{s_1 s_2 \tau_1}\bigr)^{-1}
   = s_1 s_2 \tau_1 (s_1 s_2 s_1)^{-1}
   = s_1 s_2 \tau_1 s_1 s_2 s_1 \\[-2pt]
&\qquad \qquad \qquad \ \ \ \ \hspace*{0.04cm}
   = \tau_2 s_1 s_2 s_1 s_2 s_1
   = \tau_2 s_2 (s_2 s_1)^3
   = b(s_2 s_1)^3, \\[4pt]
& s_2 s_1 \tau_1\bigl(\overline{s_2 s_1 \tau_1}\bigr)^{-1}
   = s_2 s_1 \tau_1 (s_2)^{-1}
   = s_2 s_1 \tau_1 s_2=c, \\[4pt]
& s_1 s_2 s_1 \tau_1\bigl(\overline{s_1 s_2 s_1 \tau_1}\bigr)^{-1}
   = s_1 s_2 s_1 \tau_1 s_2 s_1
   = (s_1 s_2)^3 s_2 \tau_2=(s_1 s_2)^3 b.
\end{aligned}
\]

\noindent Similarly, the nontrivial Schreier generators arising from $\tau_2$ are (after cancellation and simplifying by the relations of $ST_3$)
\[
s_1 s_2 \tau_2 s_1=d,\qquad s_2 \tau_2=b,\qquad s_2 s_1 \tau_2 s_1 s_2 s_1 = a,\qquad s_1 s_2 s_1 \tau_2 s_1 s_2 = a.
\]

\noindent By elementary computation using the relations of the singular twin group $ST_3$ one checks the identities
\[
(s_2 s_1)^3=\bigl((s_1 s_2)^3\bigr)^{-1},\qquad
(s_1 s_2)^3=(s_1 s_2 \tau_2 s_1)(s_2 s_1 \tau_1 s_2)^{-1},
\]
so the Schreier generators above reduce to the following four  elements of $SPT_3$:
\[
a=s_1\tau_1,\qquad b=s_2\tau_2,\qquad c=s_2 a s_2,\qquad d=s_1 b s_1.
\]
Hence, $SPT_3$ is generated by $a,b,c,d$, as required.

\subsection{Topological interpretation} 

Recall that a fundamental theorem in knot theory states that any braid can be closed in a standard manner to yield a knot or a link in  the $3$-dimensional sphere $S^{3}$ \cite{Alexander1923}. Similarly, the closure of a singular braid yields a singular link, that is, a link which  can be represented by a planar diagram that is allowed to have a finite number of transverse double points called singularities.  
In the same spirit, an element of the twin group can be closed to define a doodle \cite{FT}.  
Formally, a doodle is an immersion of a finite disjoint union of circles into the $2$-dimensional sphere $S^{2}$, considered up to homotopy without creating triple points  \cite{Khovanov1997}.  
In other words, two doodles are regarded as equivalent if they can be related by a finite sequence of the following two local moves:  

\begin{itemize}
    \item \textbf{Move $D_{1}$}: Creation or elimination of a monogon (a small simple loop with no intersections).
    \item \textbf{Move $D_{2}$}: Creation or elimination of a bigon (two arcs forming a simple lens-shaped region).
\end{itemize}

\begin{figure}[h]
\centering
\includegraphics[width=5cm,height=1.8cm]{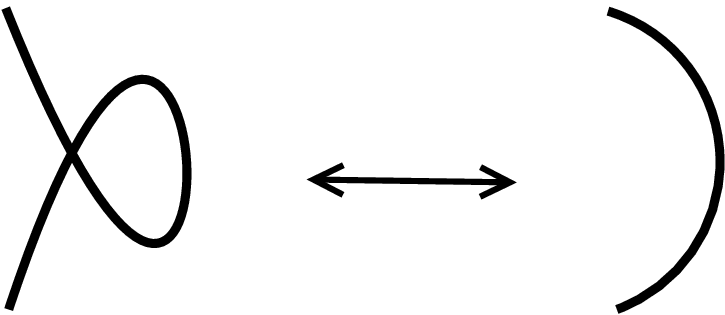}\hspace{1cm}
\includegraphics[width=5.5cm,height=1.8cm]{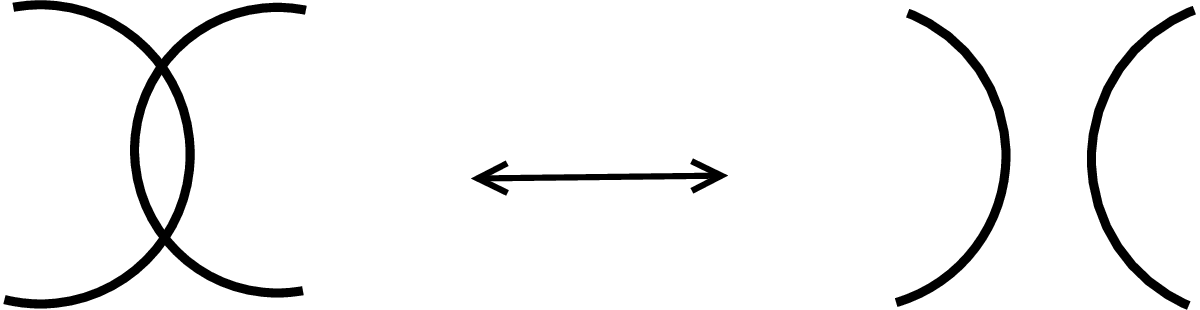}
\caption{The local moves $D_{1}$ (left) and $D_{2}$ (right).}
\label{figure4}
\end{figure}

These moves play the same role for doodles as the Reidemeister moves do for classical knots and links, providing a combinatorial description of their equivalence classes.

Likewise, a topological interpretation of the singular twin group can be formulated.  
Indeed, the closure of a singular twin element can be viewed as a \emph{singular doodle}: an immersion of a $4$-valent graph, possibly together with a collection of disjoint circles, into $S^{2}$, see Figure \ref{figure5}.  
Such singular doodles exhibit two types of singular features:
\begin{enumerate}
    \item \textbf{Transverse double points}, corresponding to intersections of edges of the immersed graph, and
    \item \textbf{$4$-valent vertices}, representing the vertices of the underlying graph.
\end{enumerate}

\begin{figure}[h]
\centering
\includegraphics[width=4cm,height=4cm]{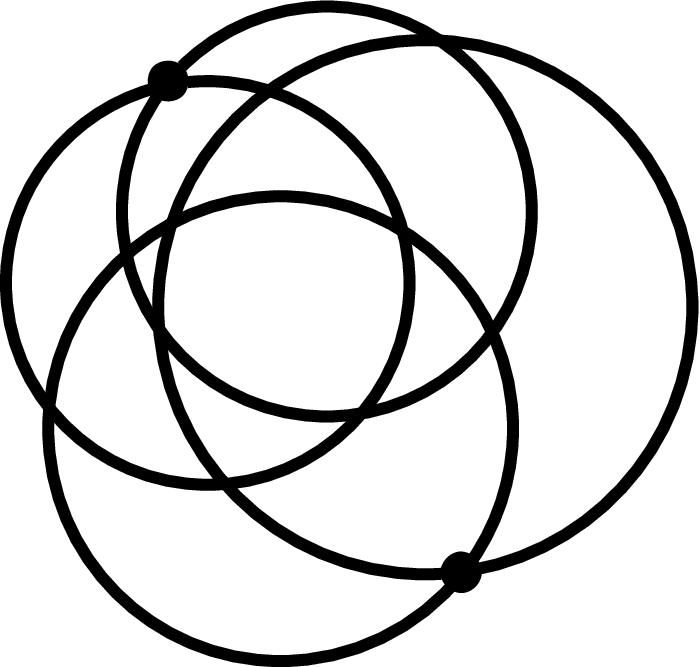}
\caption{A singular doodle obtained   as an  embedding of a 4-valent graph. The underlying graph is a disjoint union of 2 two-bouquet graphs.}
\label{figure5}
\end{figure}

Two singular doodles are said to be equivalent if one can be transformed into the other by a finite sequence of local moves of type $D_{1}$ and $D_{2}$ (Figure~\ref{figure4}), together with moves $D_{3}$ and $D_{4}$ (Figure~\ref{figure6}), which extend the classical doodle moves to configurations involving vertices.

\begin{figure}[h]
\centering
\includegraphics[width=5cm,height=2cm]{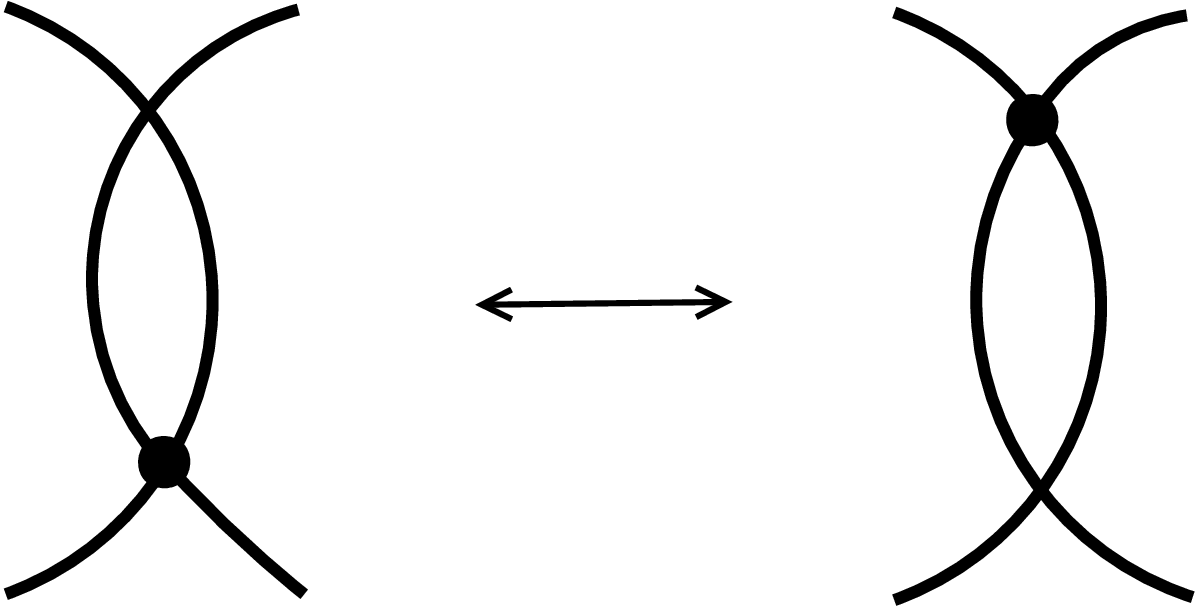}\hspace{1cm}
\includegraphics[width=6cm,height=2cm]{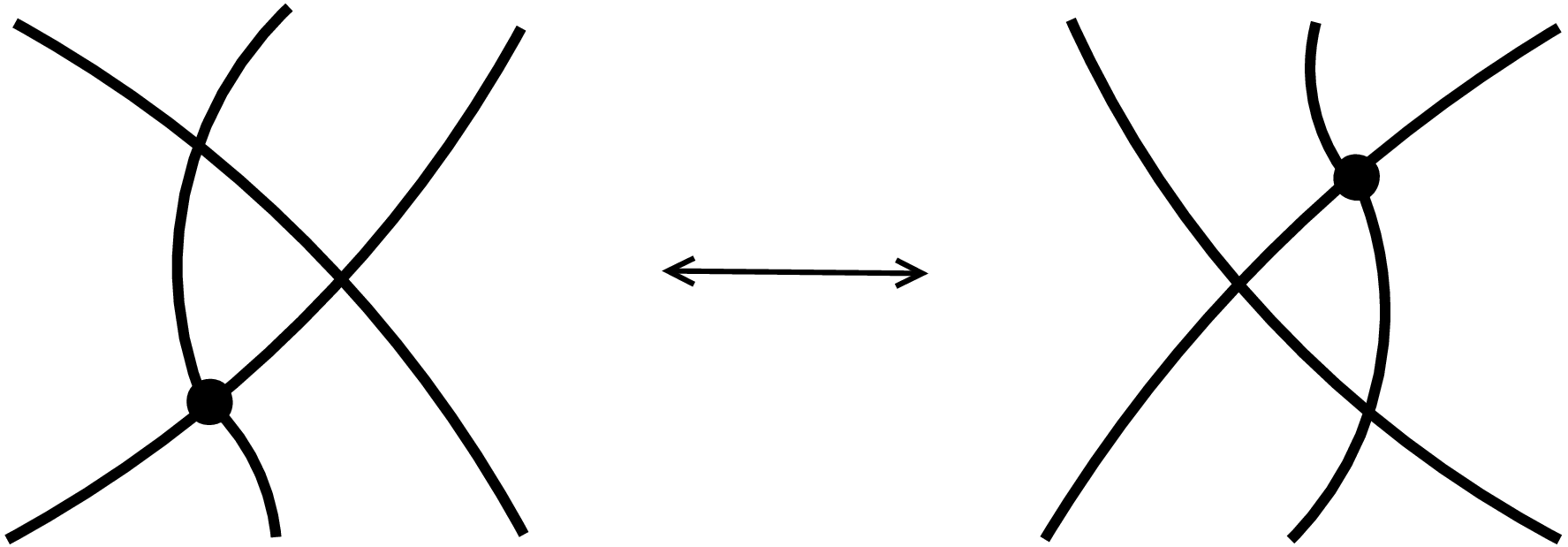}
\caption{The local moves $D_{3}$ (left) and $D_{4}$ (right).}
\label{figure6}
\end{figure}

\vspace*{0.1cm}

Now, after introducing the singular twin group from both algebraic and topological perspectives, we aim to construct representations of this group.  

\begin{question}
What are the possible complex homogeneous $k$-local representations of $ST_n$?
\end{question}

We answer this question in the following sections for $k=2$.

\section{Classification of homogeneous $2$-Local representations of $ST_n$} %%%%%% Section 4 %%%%%%%%

In this section we classify all complex homogeneous $2$-local representations of the singular twin group $ST_n$ for all $n\geq 2$. First of all, we consider the case $n=2$, which is a special case.

\vspace*{0.1cm}

\begin{theorem} \label{n2}
Let $\Gamma: ST_2 \rightarrow \mathrm{GL}_2(\mathbb{C})$ be a complex homogeneous $2$-local representation of $ST_2$. Then, $\Gamma$ is equivalent to one of the following six representations.
\begin{itemize}
\item[(1)] $\Gamma_1: ST_2 \rightarrow \mathrm{GL}_2(\mathbb{C})$ such that $$\Gamma_1(s_1)=\left(
\begin{array}{cc}
 -\sqrt{1-b c} & b \\
 c & \sqrt{1-b c} \\
\end{array}
\right)
\ \text{and} \ \Gamma_1(\tau_1)= \left(
\begin{array}{cc}
 w & x \\
 \frac{c x}{b} & \frac{b w+2 x \sqrt{1-b c}}{b} \\
\end{array}
\right),$$ where $b,c,w,x \in \mathbb{C}, b\neq 0.$ \vspace*{0.1cm}
\item[(2)] $\Gamma_2: ST_2 \rightarrow \mathrm{GL}_2(\mathbb{C})$ such that $$\Gamma_2(s_1)=\left(
\begin{array}{cc}
 \sqrt{1-b c} & b \\
 c & -\sqrt{1-b c} \\
\end{array}
\right)
\ \text{and} \ \Gamma_2(\tau_1)=\left(
\begin{array}{cc}
 w & x \\
 \frac{c x}{b} & \frac{b w-2 x \sqrt{1-b c}}{b} \\
\end{array}
\right),$$ where $b,c,w,x \in \mathbb{C}, b\neq 0.$ \vspace*{0.1cm}
\item[(3)] $\Gamma_3: ST_2 \rightarrow \mathrm{GL}_2(\mathbb{C})$ such that $$\Gamma_3(s_1)=\left(
\begin{array}{cc}
 -1 & 0 \\
 c & 1 \\
\end{array}
\right)
\ \text{and} \ \Gamma_3(\tau_1)=\left(
\begin{array}{cc}
 w & 0 \\
 y & \frac{c w+2 y}{c} \\
\end{array}
\right),$$ where $c,w,y \in \mathbb{C}, c\neq 0.$ \vspace*{0.1cm}
\item[(4)] $\Gamma_4: ST_2 \rightarrow \mathrm{GL}_2(\mathbb{C})$ such that $$\Gamma_4(s_1)=\left(
\begin{array}{cc}
 1 & 0 \\
 c & -1 \\
\end{array}
\right)
\ \text{and} \ \Gamma_4(\tau_1)=\left(
\begin{array}{cc}
 w & 0 \\
 y & \frac{c w-2 y}{c} \\
\end{array}
\right),$$ where $c,w,y \in \mathbb{C}, c\neq 0.$ \vspace*{0.1cm}
\item[(5)] $\Gamma_5: ST_2 \rightarrow \mathrm{GL}_2(\mathbb{C})$ such that $$\Gamma_5(s_1)=\left(
\begin{array}{cc}
 -1 & 0 \\
 0 & 1 \\
\end{array}
\right)
\ \text{and} \ \Gamma_5(\tau_1)=\left(
\begin{array}{cc}
 w & 0 \\
 0 & z \\
\end{array}
\right),$$ where $w,z \in \mathbb{C}.$ \vspace*{0.1cm}
\item[(6)] $\Gamma_6: ST_2 \rightarrow \mathrm{GL}_2(\mathbb{C})$ such that $$\Gamma_6(s_1)=\left(
\begin{array}{cc}
 1 & 0 \\
 0 & 1 \\
\end{array}
\right)
\ \text{and} \ \Gamma_6(\tau_1)=\left(
\begin{array}{cc}
 w & x \\
 y & z \\
\end{array}
\right),$$ where $w,x,y,z \in \mathbb{C}.$ \vspace*{0.1cm}
\end{itemize}
\end{theorem}

\begin{proof}
Set $$\Gamma(s_1)=\left(
\begin{array}{cc}
 a & b \\
 c & d \\
\end{array}
\right)
\ \text{and} \ \Gamma(\tau_1)=\left(
\begin{array}{cc}
 w & x \\
 y & z \\
\end{array}
\right),$$ where $a,b,c,d,w,x,y,z \in \mathbb{C}$ such that $ad-bc\neq 0$ and $wz-xy\neq 0$. The defining relations of the group $ST_2$ are $s_1^2=1$ and $s_1\tau_1=\tau_1s_1$. Consequently, we obtain $\Gamma(s_1)^2=1$ and $\Gamma(s_1)\Gamma(\tau_1)=\Gamma(\tau_1)\Gamma(s_1)$. Using these two relations, we derive the following seven equations.
\begin{equation} \label{eq15}
a^2+bc=1
\end{equation}
\begin{equation}\label{eq16}
ab+bd=0
\end{equation}
\begin{equation}\label{eq17}
ac+cd=0
\end{equation}
\begin{equation}\label{eq18}
bc+d^2=1
\end{equation}
\begin{equation}\label{eq19}
-cx+by=0
\end{equation}
\begin{equation}\label{eq20}
bw-ax+dx-bz=0
\end{equation}
\begin{equation}\label{eq21}
cw-ay+dy-cz=0
\end{equation}
We consider the following two cases.
\begin{itemize}
\item[(1)] The case $b=0$. From Equations (\ref{eq15}) and (\ref{eq18}), we get that $a^2=d^2=1$, and so $a=\pm 1$ and $d=\pm 1$. We consider each subcase separately.
\begin{itemize}
\item[•] If $a=d=1$, then $c=0$ by Equation (\ref{eq17}) and so $\Gamma$ is equivalent to $\Gamma_6$ in this case. 
\item[•] If $a=d=-1$, then $c=0$ by Equation (\ref{eq17}) and so $\Gamma$ is equivalent to $\Gamma_6$ in this case. 
\item[•] If $a=1$ and $d=-1$, then $x=0$ by Equation (\ref{eq20}) and so we have $\Gamma$ is equivalent to $\Gamma_5$ if $c=0$ and $\Gamma$ is equivalent to $\Gamma_4$ if $c\neq 0$.
\item[•] If $a=-1$ and $d=1$, then $x=0$ by Equation (\ref{eq20}) and so we have $\Gamma$ is equivalent to $\Gamma_5$ if $c=0$ and $\Gamma$ is equivalent to $\Gamma_3$ if $c\neq 0$.
\end{itemize}
\item[(2)] The case $b \neq 0$. From Equations (\ref{eq15}) and (\ref{eq18}), we get that $a^2=d^2=1-bc$ and, by Equation (\ref{eq16}), we get that $a=-d$. So, $a=\pm \sqrt{1-bc}$ and $d=\mp \sqrt{1-bc}$. We consider each subcase separately.
\begin{itemize}
\item[•] If $a=-\sqrt{1-bc}$ and $d=\sqrt{1-bc}$, then, using Equations (\ref{eq20}) and (\ref{eq21}), we get that $\Gamma$ is equivalent to $\Gamma_1$. 
\item[•] If $a=\sqrt{1-bc}$ and $d=-\sqrt{1-bc}$, then, using Equations (\ref{eq20}) and (\ref{eq21}), we get that $\Gamma$ is equivalent to $\Gamma_2$. 
\end{itemize}
\end{itemize}
\end{proof}

We now consider the case $n\geq 3$ and we classify all complex homogeneous $2$-local representations of $ST_n$, for all $n\geq 3$. 

\begin{theorem} \label{Theo}
Consider $n\geq 3$ and let $\Theta: ST_n \rightarrow \mathrm{GL}_n(\mathbb{C})$ be a complex homogeneous $2$-local representation of $ST_n$. Then, $\Theta$ is equivalent to one of the following five representations.
\begin{itemize}
\item[(1)] $\Theta_1: ST_n \rightarrow \mathrm{GL}_n(\mathbb{C})$ such that
$$\Theta_1(s_i)=\left( \begin{array}{c|@{}c|c@{}}
   \begin{matrix}
     I_{i-1} 
   \end{matrix} 
      & 0 & 0 \\
      \hline
    0 &\hspace{0.2cm} \begin{matrix}
  0 & b \\
 \frac{1}{b} & 0 \\
   		\end{matrix}  & 0  \\
\hline
0 & 0 & I_{n-i-1}
\end{array} \right)
 \ \text{and } \ \Theta_1(\tau_i) =\left( \begin{array}{c|@{}c|c@{}}
   \begin{matrix}
     I_{i-1} 
   \end{matrix} 
      & 0 & 0 \\
      \hline
    0 &\hspace{0.2cm} \begin{matrix}
  w & x \\
 \frac{x}{b^2} & w \\
   		\end{matrix}  & 0  \\
\hline
0 & 0 & I_{n-i-1}
\end{array} \right),$$ where $b,w,x\in \mathbb{C},b\neq 0$, $1\leq i\leq n-1.$\vspace*{0.1cm}
\item[(2)] $\Theta_2: ST_n \rightarrow \mathrm{GL}_n(\mathbb{C})$ such that
$$\Theta_2(s_i)=\left( \begin{array}{c|@{}c|c@{}}
   \begin{matrix}
     I_{i-1} 
   \end{matrix} 
      & 0 & 0 \\
      \hline
    0 &\hspace{0.2cm} \begin{matrix}
 -\sqrt{1-b c} & b \\
 c & \sqrt{1-b c} \\
   		\end{matrix}  & 0  \\
\hline
0 & 0 & I_{n-i-1}
\end{array} \right)
 \ \text{and } \ \Theta_2(\tau_i)=I_n,$$ where $b,c \in \mathbb{C}$, $1\leq i\leq n-1.$\vspace*{0.1cm}
\item[(3)] $\Theta_3: ST_n \rightarrow \mathrm{GL}_n(\mathbb{C})$ such that
$$\Theta_3(s_i)=\left( \begin{array}{c|@{}c|c@{}}
   \begin{matrix}
     I_{i-1} 
   \end{matrix} 
      & 0 & 0 \\
      \hline
    0 &\hspace{0.2cm} \begin{matrix}
 \sqrt{1-b c} & b \\
 c & -\sqrt{1-b c} \\
   		\end{matrix}  & 0  \\
\hline
0 & 0 & I_{n-i-1}
\end{array} \right)
 \ \text{and } \ \Theta_3(\tau_i) =I_n,$$ where $b,c \in \mathbb{C}$, $1\leq i\leq n-1.$\vspace*{0.1cm}
\item[(4)] $\Theta_4: ST_n \rightarrow \mathrm{GL}_n(\mathbb{C})$ such that
$$\Theta_4(s_i)=\left( \begin{array}{c|@{}c|c@{}}
   \begin{matrix}
     I_{i-1} 
   \end{matrix} 
      & 0 & 0 \\
      \hline
    0 &\hspace{0.2cm} \begin{matrix}
 -1 & 0 \\
 0 & -1 \\
   		\end{matrix}  & 0  \\
\hline
0 & 0 & I_{n-i-1}
\end{array} \right)
 \ \text{and } \ \Theta_4(\tau_i) =I_n,$$ where $1\leq i\leq n-1.$\vspace*{0.1cm}
\item[(5)] $\Theta_5: ST_n \rightarrow \mathrm{GL}_n(\mathbb{C})$ such that
$$\Theta_5(s_i)=I_n
 \ \text{and } \ \Theta_5(\tau_i) =I_n,$$ where $1\leq i\leq n-1.$\vspace*{0.1cm}

\end{itemize}
\end{theorem}
\begin{proof}
The proof follows in a similar manner to that of Theorem \ref{n2}.
\end{proof}

\vspace*{0.1cm}

\section{Irreducibility of the homogeneous $2$-local representations of $ST_n$} %%%%%%%%%% Section 5 %%%%%%%%%%%%%

In this section, we study the irreducibility of the complex homogeneous $2$-local representations of the singular twin group $ST_n$ for all $n\geq 2$. We start by the case $n=2$, which is a special case.

\begin{theorem}
Let $\Gamma: ST_2 \rightarrow \mathrm{GL}_2(\mathbb{C})$ denote a complex homogeneous $2$-local representation of $ST_2$. Then, $\Gamma$ is reducible. 
\end{theorem}

\begin{proof}
Theorem \ref{n2} yields that $\Gamma$ is equivalent to one of the six representations $\Gamma_j, 1\leq j \leq 6$. We consider each case separately.
\begin{itemize}
\item[(1)] If $\Gamma$ is equivalent to $\Gamma_1$ or $\Gamma_2$, then we have the following two subcases.
\item[•] In the case $c=0$, $e_1$ is a common eigenvector of both $\Gamma(s_1)$ and $\Gamma(\tau_1)$, and hence $\Gamma$ is reducible.
\item[•] In the case $c\neq 0$, $\left(-\frac{\sqrt{1-b c}+1}{c},1\right)$ is a common eigenvector of both $\Gamma(s_1)$ and $\Gamma(\tau_1)$, and hence $\Gamma$ is reducible.
\item[(2)] If $\Gamma$ is equivalent to $\Gamma_3$, $\Gamma_4$ or $\Gamma_5$, then $e_2$ is a common eigenvector of both $\Gamma(s_1)$ and $\Gamma(\tau_1)$, and hence $\Gamma$ is reducible.
\item[(3)] If $\Gamma$ is equivalent to $\Gamma_6$, then every eigenvector of $\Gamma(\tau_1)$ is invariant under $\Gamma(s_1)$, and hence $\Gamma$ is reducible.
\end{itemize}
\end{proof}

\begin{theorem} \label{Thirrr}
Consider $n\geq 3$ and let $\Theta: ST_n \rightarrow \mathrm{GL}_n(\mathbb{C})$ denote a complex homogeneous $2$-local representation of $ST_n$. By Theorem \ref{Theo}, $\Theta$ is equivalent to one of the five representations $\Theta_i, 1\leq i \leq 5.$ The following hold true.
\begin{itemize}
\item[(1)] If $\Theta$ is equivalent to $\Theta_1$, then $\Theta$ is irreducible if and only if $w+\frac{x}{b}\neq 1$.
\item[(2)] If $\Theta$ is equivalent to $\Theta_2$, then $\Theta$ is reducible to the degree $n-1$. Furthermore, by putting $a = -\sqrt{1-bc}$, we have the following cases.
\begin{itemize}
\item[•] If $n=3$, then the $(n-1)$-composition factor, namely $\Theta'$, of $\Theta$ is irreducible if and only if $a \notin \{\pm 1,\pm i\sqrt{3}\}$.
\item[•] If $n\geq 4$, then the $(n-1)$-composition factor, namely $\Theta'$, of $\Theta$ is irreducible if and only if $a \notin \{\pm 1\}$ and $a$ is not a root of $$P(t) = 4(1+t^2) + \frac{(1-t)^4}{2t} \left( 1 - \left( \frac{1-t}{1+t} \right)^{n-4} \right).$$
\end{itemize} 
\item[(3)] If $\Theta$ is equivalent to $\Theta_3$, then $\Theta$ is reducible to the degree $n-1$. Furthermore, by putting $a = \sqrt{1-bc}$, we have the following cases.
\begin{itemize}
\item[•] If $n=3$, then the $(n-1)$-composition factor, namely $\Theta'$, of $\Theta$ is irreducible if and only if $a \notin \{\pm 1,\pm i\sqrt{3}\}$.
\item[•] If $n\geq 4$, then the $(n-1)$-composition factor, namely $\Theta'$, of $\Theta$ is irreducible if and only if $a \notin \{\pm 1\}$ and $a$ is not a root of $$P(t) = 4(1+t^2) + \frac{(1-t)^4}{2t} \left( 1 - \left( \frac{1-t}{1+t} \right)^{n-4} \right).$$
\end{itemize} 
\item[(4)] If $\Theta$ is equivalent to $\Theta_4$ or $\Theta_5$, then $\Theta$ is a direct sum of $1$-dimensional representations.
\end{itemize} 
\end{theorem}

\begin{proof}
We examine each case individually in what follows, except for the proofs of (4) and (5), which are straightforward.
\begin{itemize}
\item[(1)] Suppose that $\Theta$ is equivalent to $\Theta_1$. Consider the diagonal matrix defined by $P=\text{diag}(b^{1-n},b^{2-n},\ldots, b,1)$, where $\text{diag}(r_1, r_2,\dots, r_n)$ is a diagonal $n\times n$ matrix with $r_{ii}=r_i$. Consider the equivalent representation $\hat{\Theta}$ of $\Theta$ given by: $\hat{\Theta}(s_i)=P^{-1}\Theta(s_i)P$ and $\hat{\Theta}(\tau_i)=P^{-1}\Theta(\tau_i)P$ for all $1\leq i \leq n-1$. Direct computations give that 
$$\hat{\Theta}(s_i) =\left( \begin{array}{c|@{}c|c@{}}
   \begin{matrix}
     I_{i-1} 
   \end{matrix} 
      & 0 & 0 \\
      \hline
    0 &\hspace{0.2cm} \begin{matrix}
   		0 & 1\\
   		1 & 0
   		\end{matrix}  & 0  \\
\hline
0 & 0 & I_{n-i-1}
\end{array} \right)$$
and
$$\hat{\Theta}(\tau_i) =\left( \begin{array}{c|@{}c|c@{}}
   \begin{matrix}
     I_{i-1} 
   \end{matrix} 
      & 0 & 0 \\
      \hline
    0 &\hspace{0.2cm} \begin{matrix}
   		 w & \frac{x}{b}\\
   		\frac{x}{b} & w
   		\end{matrix}  & 0  \\
\hline
0 & 0 & I_{n-i-1}
\end{array} \right),$$
$\text{where} \hspace*{0.15cm} w, b, x \in \mathbb{C},b \neq 0, \hspace*{0.15cm} \text{for} \hspace*{0.15cm} 1\leq i\leq n-1$. The representation $\hat{\Theta}$ has the same form as the representation $\rho_3$ obtained in \cite[Theorem~30]{Nasiss}. Referring to the results in that paper, we obtain that our representation $\hat{\Theta}$ is irreducible if and only if $w + \tfrac{x}{b} \neq 1$, and consequently the same holds for $\Theta$.
\item[(2)] Suppose that $\Theta$ is equivalent to $\Theta_2$ and set $a=-\sqrt{1-bc}$. The restriction of the representation $\Theta$ to $T_n$ in this case has the same form as the representation $\xi_1$ obtained in \cite[Theorem~5]{Mayasi20251}. Referring to the results in that paper, and since $\Theta(\tau_i)=I_n$ for all $1\leq i \leq n-1$, we obtain that our representation $\Theta$ is reducible to the degree $n-1$ and the following cases occur.
\begin{itemize}
\item[•] If $n=3$, then the $(n-1)$-composition factor, namely $\Theta'$, of $\Theta$ is irreducible if and only if $a \notin \{\pm 1,\pm i\sqrt{3}\}$.
\item[•] If $n\geq 4$, then the $(n-1)$-composition factor, namely $\Theta'$, of $\Theta$ is irreducible if and only if $a \notin \{\pm 1\}$ and $a$ is not a root of $$P(t) = 4(1+t^2) + \frac{(1-t)^4}{2t} \left( 1 - \left( \frac{1-t}{1+t} \right)^{n-4} \right).$$
\end{itemize}
\item[(3)] In the case $\Theta$ is equivalent to $\Theta_3$, the argument proceeds as in (2), this time taking $a=\sqrt{1-bc}$.
\end{itemize}
\end{proof}

\section{Future work} %%%%%%%%%% Section 6 %%%%%%%%%%%%%
In this section, we provide  ideas that could be considered as future work.
\begin{itemize}
\item[(1)] One of the important questions that could be addressed for any constructed group is its linearity. A group is said to be linear if it admits a faithful representation. So, the first issue that could be considered for the future is to study the faithfulness of the classified representations.
\item[(2)] In addition to classifying and analyzing $k$-local representations of the singular twin group, we also encourage the construction of new non-local representations of this group and the investigation of their properties, such as irreducibility and faithfulness.
\item[(3)] Inspired by the relationship between the Burau representation of the braid group and the Alexander polynomial for knots, we propose a future study to investigate whether representations of the twin group and the singular twin group can be used to define analogous invariants for doodles and singular doodles.
\end{itemize}
\vspace{.5cm}
\subsection*{Acknowledgement}
The second   author was    supported by  United Arab Emirates University under  UPAR grant
 $\# G00005447.$
\vspace{0.2cm}

%%%%%%%%%%% REFERENCES %%%%%%%%%%%%%%%%

\end{document}